\documentclass[a4paper,11pt]{article}
\usepackage{amsfonts}
\usepackage{amsmath}
\usepackage{amsthm}
\usepackage{amssymb}
\newtheorem{theorem}{Theorem}[section]
\newtheorem{lemma}[theorem]{Lemma}

\theoremstyle{definition}

\theoremstyle{corollary}

\theoremstyle{remark}

\numberwithin{equation}{section}

\newcommand{\br}{\mathbb R}
\newcommand{\bn}{\mathbb N}

\title{A family of  singular ordinary  differential equations
of third order with an integral boundary condition}
\author{
{Mahdi Boukrouche\thanks{
Address : Lyon University, F-42023 Saint-Etienne,
Institut Camille Jordan CNRS UMR 5208,
23 rue  Paul Michelon 42023 Saint-Etienne Cedex 2, France.
Mahdi.Boukrouche@univ-st-etienne.fr}}
\and
Domingo A. Tarzia\thanks{
Adress: Departamento de Matem\'atica-CONICET,
FCE, Univ. Austral, Paraguay 1950, S2000FZF Rosario, Argentina.
DTarzia@austral.edu.ar}
}
\date{}

\begin{document}
\maketitle \normalsize

%    Abstract is require
\begin{abstract}
We establish in this paper the equivalence between a Volterra
integral equation of second kind and a singular  ordinary differential
equation of third order with two initial conditions and an integral boundary
condition, with a real parameter.
This equivalence allow us to obtain the solution to some problems for nonclassical heat
equation, the continuous dependence of the solution with respect to the parameter and 
the corresponding explicit solution to the considered problem.

\smallskip


\smallskip

\end{abstract}

\section{Introduction}
We consider the following family of singular ordinary differential equations of third order  with an integral
boundary condition, indexed by a parameter $\lambda \in \mathbb{R}$ given by
\begin{eqnarray}\label{edo}
\left.
\begin{array}{ll}
 y^{(3)}(t) - \lambda^{2}y(t) = {\lambda\over 2\sqrt{\pi}}{1\over t^{3/2}},  \quad  t>0, \\ \\
 y(0)=1, \quad y'(0)= 0, \quad  y^{(2)}(1) = -{\lambda\over\sqrt{\pi}} + \lambda^{2}\int_{0}^{1} y(t) dt,
\end{array}
\right\}
\end{eqnarray}
where $y^{(n)}$ denotes the $n-$derivative of the function $y$.

Singular boundary value problems arise very frequently in fluid mechanics and in  other branches of applied
mathematics. There are results on the existence and asymptotic estimates of solutions for third order ordinary
differential equations with singularly perturbed boundary value problems, which depend on a small positive parameter
see for example \cite{Du2007, DuGeZh2005, Lin2013}, on third order ordinary
differential equations with singularly perturbed boundary value problems and with nonlinear coefficients or
boundary conditions see for example \cite{BeBoBo2008, Ch2015, LiUmAnKo2007, Zh2011}, on
 third order ordinary
differential equations with nonlinear  boundary value problems see for example \cite{DuGeLi2004, LiDoLi2008},
 on existence results for third order ordinary
differential equations see for example \cite{Du2011, JiLi2007},
and particularly third order ordinary
differential equations with  integral boundary conditions see for example
\cite{AlNtMeAhAl2015, BoBoMaBe2009, BoBoBeMa2014, GrKo2009, GuLiLi2012, PaXiCa2015, SuLi2010, WaGe2007, Zh2014}

In the last years there are several papers which consider integral or nonlocal boundary conditions on different
branches of applications, e.g. for the heat equations see for example \cite{CaLi1990, DaHu2007, De2005, De2007, HaThLeIv2014,
Ka2013, Li1999, Ma2009, MaVi2009A, MaVi2009B, MeBo2007}, for the wave equations \cite{MaWa2012}, for the  second
order ordinary differential equations see for example \cite{Bo2009, LiLILi2014, LvFe2016, ToDi2016, ZhFe2015, ZhFeGe2009, ZhXu2016},
for the fourth order ordinary differential equations  see for example \cite{SuXi2016, ZhGe2009-1}, for higher order
 ordinary differential equations see for example \cite{JiGuYa2015}, for fractional differential equations see for example
 \cite{HeLu2016, LiLiWu2015, WaGu2016}.

Our goal is to prove in Section 2 that the system (\ref{edo}) is equivalent to the following Volterra integral equation of second kind

\begin{eqnarray}\label{Vol}
 y(t)= 1- {2\lambda\over \sqrt{\pi}}\int_{0}^{t} y(\tau) \sqrt{t-\tau} d\tau, \quad t>0, \quad (\lambda \in \br)
\end{eqnarray}
which allows us to obtain the solution to some problems for nonclassical heat equation for any real parameter $\lambda$
(see \cite{BeTaVi2000, BoTA2016b, BoTA2016a, CeTaVi2015, SaTaVi2011, TaVi1998, Vi1986}).

In Section 3, we establish the dependence  of the family  of singular ordinary differential equations of third order  (\ref{edo})
with respect to the parameter $\lambda \in \br$ by using the equivalence with the Volterra integral equation (\ref{Vol}).

\section{Equivalence and existence results}

Preliminary, we give some results useful in the next sections.

\begin{lemma}\label{lem1}
 We have the following properties
 \begin{eqnarray}\label{eq3}
 \int_{0}^{t}\left(\int_{0}^{\tau} y(\xi) d\xi\right)d\tau = \int_{0}^{t} y(\tau)(t-\tau)  d\tau
 \end{eqnarray}
\begin{eqnarray}\label{eq4}
 \int_{0}^{t}\left(\int_{0}^{\xi}y(\tau)(t-\tau)d\tau\right) d\xi
 =\int_{0}^{t}  y(\tau)(t-\tau)^{2} d\tau
\end{eqnarray}
\begin{eqnarray}\label{eq5}
 \int_{0}^{t} {y(\tau)\over \sqrt{t-\tau}} d\tau = 2\sqrt{t} + 2 \int_{0}^{t} y'(\tau)\sqrt{t-\tau} d\tau
\end{eqnarray}
\begin{eqnarray}\label{eq6}
  \int_{\sigma}^{t} {\sqrt{\tau-\sigma}\over \sqrt{t-\tau}} d\tau ={\pi\over 2}(t-\sigma),
\end{eqnarray}
\begin{eqnarray}\label{eq7}
   \int_{\sigma}^{t} {d\tau\over \sqrt{t-\tau}\sqrt{\tau-\sigma}}= \pi.
\end{eqnarray}
\end{lemma}
\begin{proof}
 The first three properties (\ref{eq3})-(\ref{eq5}) follow from the the simple integration process.
 To prove (\ref{eq6}) we use the change of variable $\tau= \sigma + (t-\sigma)\xi$ then we obtain
 \begin{eqnarray*}
  \int_{\sigma}^{t} {\sqrt{\tau-\sigma}\over \sqrt{t-\tau}} d\tau
 &=&  (t-\sigma)\int_{0}^{1} \sqrt{{\xi\over 1-\xi}}d\xi
  =(t-\sigma)\int_{0}^{1}\xi^{{3\over 2}-1}(1-\xi)^{{1\over 2}-1} d\xi
  \nonumber\\
  &=& (t-\sigma) B({3\over 2}, {1\over 2})= (t-\sigma){\Gamma({3\over 2})\Gamma(({1\over 2})\over \Gamma(2)}=
  {\pi\over 2}(t-\sigma),
\end{eqnarray*}
where $B$ and $\Gamma$ are the known Beta and Gamma functions defined by
\begin{eqnarray*}
 B(x, y) &=& \int_{0}^{1} t^{x-1} (1-t)^{y-1} dt, \quad x>0, \quad y>0, \nonumber\\
 \Gamma(x) &=& \int_{0}^{+\infty} t^{x-1} e^{-t} dt, \quad x>0,
\end{eqnarray*}
with the well known relations
\begin{eqnarray*}
 B(x , y) ={\Gamma(x)\Gamma(y)\over \Gamma(x+y)}, \quad \Gamma(x+1)=x\Gamma(x) \quad\forall x>0,
 \quad \Gamma({1\over 2})= \sqrt{\pi}, \quad \Gamma(n+1)=n! \quad\forall n\in \bn.
\end{eqnarray*}

To prove (\ref{eq7}) we use the same change of variable, so we obtain
 \begin{eqnarray*}
   \int_{\sigma}^{t} {d\tau\over \sqrt{t-\tau}\sqrt{\tau-\sigma}}
   =\int_{0}^{1} {d\xi\over \sqrt{\xi(1-\xi)}}= B({1\over 2}, {1\over 2})
   = {\Gamma({1\over 2})\Gamma({1\over 2})\over \Gamma(1)}
   = \pi.
 \end{eqnarray*}
\end{proof}

\begin{theorem}\label{th2}
  $y$ is a solution to the singular ordinary differential equation {\rm(\ref{edo})} if and only if
 $y$ is a solution  to the Volterra integral equation {\rm(\ref{Vol})}.
\end{theorem}
\begin{proof}
 Firstly, we consider that $y$ is  a solution to the singular ordinary differential
 equation {\rm(\ref{edo})}. Then, by using an integration in variable $t$
 we obtain
  \begin{eqnarray}\label{eq8}
  y^{(2)}(t) = y^{(2)}(0)+ \lambda^{2} \int_{0}^{t} y(\tau)d\tau  -{\lambda\over \sqrt{\pi t}}, \quad t>0.
  \end{eqnarray}
  thus
  $$ y^{(2)}(1) = y^{(2)}(0)+ \lambda^{2} \int_{0}^{1} y(\tau)d\tau  -{\lambda\over \sqrt{\pi }}.$$

 And using the integral  boundary condition
 $$y^{(2)}(1)= -{\lambda\over\sqrt{\pi}} + \lambda^{2}\int_{0}^{1} y(t) dt $$
so $y^{(2)}(0)= 0$.  Thus taking this new condition into  account, from (\ref{eq8})
by  using an integration in variable $t$,  the condition $y'(0)=0$ and
(\ref{eq3}) we get
\begin{eqnarray}\label{eq9}
  y'(t) &=& \lambda^{2}\int_{0}^{t}\left( \int_{0}^{\tau} y(\sigma) d\sigma \right) d\tau
   -{2\lambda\sqrt{t}\over \sqrt{\pi}}
   \nonumber\\
   &=& \lambda^{2}\int_{0}^{t}y(\tau)(t-\tau) d\tau - {2\lambda\sqrt{t}\over \sqrt{\pi}}
   , \quad t>0.
  \end{eqnarray}
Finally, from (\ref{eq9}) by using a another integration in variable $t$, and the condition
$y(0)=1$, we obtain

\begin{eqnarray}\label{eq10}
  y(t) &=& 1+ \lambda^{2}   \int_{0}^{t}   \left( \int_{0}^{\tau} y(\sigma)(\tau-\sigma) d\sigma \right) d\tau
   -{4\lambda 3\over 3 \sqrt{\pi}}t^{3/2}
   \nonumber\\
   &=&  1+ \lambda^{2}\int_{0}^{t} y(\tau)(t-\tau)^{2} d\tau - {4\lambda\over 3\sqrt{\pi}}t^{3/2}
   , \quad t>0.
  \end{eqnarray}
We can not arrive directly to the Volterra  equation (\ref{Vol}), but we can define
the auxiliary function

\begin{eqnarray}\label{eq11}
 \varphi(t)= y(t) -1 + {2\lambda\over \sqrt{\pi}}  \int_{0}^{t} y(\tau)\sqrt{t-\tau}\,  d\tau
\end{eqnarray}
 and now our goal is to prove that $\varphi = 0$.  We have $\varphi(0)=0$, by using the boundary
 $y(0)=1$.

 Now, we compute the first derivative of  $\varphi$  using the property (\ref{eq5}), we get
 \begin{eqnarray}\label{eq12}
  \varphi'(t)&=& y'(t) + {\lambda\over \sqrt{\pi}}  \int_{0}^{t} {y(\tau)\over\sqrt{t-\tau}}\,  d\tau
 \nonumber\\
 &=& y'(t) + {\lambda\over \sqrt{\pi}} \left(2\sqrt{t} + 2 \int_{0}^{t} y'(\tau) \sqrt{t-\tau}\, d\tau\right)
\quad t>0.
 \end{eqnarray}

 From the other hand, by using  (\ref{eq11}), (\ref{eq9}), (\ref{eq12}) and the property  (\ref{eq6})
 we obtain
 \begin{eqnarray*}\label{eq13}
  \int_{0}^{t}  {\varphi(\tau)\over \sqrt{t-\tau}} d\tau
 & =&   \int_{0}^{t}  {y(\tau)\over \sqrt{t-\tau}} d\tau
    -2\sqrt{t} + {2\lambda\over \sqrt{\pi}}\int_{0}^{t}
    {\int_{0}^{\tau} y(\sigma)\sqrt{\tau-\sigma}d\sigma\over\sqrt{t-\tau}} d\tau
 \nonumber\\
 &=&  \int_{0}^{t}  {y(\tau)\over \sqrt{t-\tau}} d\tau
    -2\sqrt{t} +\lambda\sqrt{\pi}\int_{0}^{t} y(\tau)(t-\tau)\, d\tau
  \nonumber\\
 &=&
 \int_{0}^{t}  {y(\tau)\over \sqrt{t-\tau}} d\tau
    -2\sqrt{t} +\lambda\sqrt{\pi}\left({y'(t)\over \lambda^{2}} + {2\over \lambda\sqrt{\pi}}\sqrt{t}\right)
   \nonumber\\
 &=& \int_{0}^{t}  {y(\tau)\over \sqrt{t-\tau}} d\tau + {\sqrt{\pi}\over\lambda} y'(t)
 = {\sqrt{\pi}\over\lambda} \varphi'(t), \quad t>0.
 \end{eqnarray*}

 That is

 \begin{eqnarray}\label{eq14}
 \varphi'(t)= {\lambda\over\sqrt{\pi}}\int_{0}^{t} {\varphi(\tau)\over \sqrt{t-\tau}} d\tau,
 \quad t>0,
  \end{eqnarray}
 thus $\varphi'(0)=0$.  Therefore, we have
  \begin{eqnarray}\label{eq15}
   \varphi'(t)=  {\lambda\over\sqrt{\pi}}
   \int_{0}^{t} {\varphi(t-\tau)\over \sqrt{\tau}} d\tau,
 \quad t>0,
  \end{eqnarray}
and then we obtain

\begin{eqnarray}\label{eq15}
   \varphi^{(2)}(t)=  {\lambda\over\sqrt{\pi}}
   \int_{0}^{t} {\varphi'(t-\tau)\over \sqrt{\tau}}d\tau
   =
   {\lambda\over\sqrt{\pi}}
   \int_{0}^{t} {\varphi'(\tau)\over \sqrt{t-\tau}} d\tau,
 \quad t>0,
\end{eqnarray}

 thus $\varphi^{(2)}(0)=0$, and so  on  we obtain $\varphi^{(n)}(0)=0$ for all $n\in \mathbb{N}$,
 then this part holds.

 Secondly, we consider that $y$ is a solution of the Volterra integral equation (\ref{Vol}), then we have
 the condition $y(0)=1$ which is automatically satisfied.

 Then, by derivation of (\ref{Vol}) and by using the property (\ref{eq6})  we have

 \begin{eqnarray}\label{eq16}
  y'(t)&=& -{\lambda\over \sqrt{ \pi}}\int_{0}^{t} {y(\tau)\over \sqrt{t-\tau}} d\tau
   \nonumber\\
  &=& -{\lambda\over \sqrt{ \pi}}\left(
  2\sqrt{t} - {2\lambda\over \sqrt{\pi}}\left(
  \int_{0}^{t} {1\over \sqrt{t-\tau}}\left(
  \int_{0}^{\tau}  y(\sigma)\sqrt{\tau-\sigma}\, d\sigma
  \right)d\tau
  \right)\right)
  \nonumber\\
  &=&
  -{2\lambda\over \sqrt{ \pi}}\sqrt{t} +
  {2\lambda^{2}\over \pi}\int_{0}^{t}
  \left(
  \int_{\sigma}^{t} {\sqrt{\tau-\sigma} \over\sqrt{t-\tau}}d\tau
  \right)y(\sigma)\, d\sigma
   \nonumber\\
  &=&
  -{2\lambda\over \sqrt{ \pi}}\sqrt{t} + \lambda^{2}\int_{0}^{t}y(\sigma)(t-\sigma)\, d\sigma,
  \quad t>0,
 \end{eqnarray}
and the boundary condition $y'(0)=0$ holds. Therefore from (\ref{eq16}) we have

 \begin{eqnarray}\label{eq17}
  y^{(2)}(t)= -{\lambda\over \sqrt{\pi t}} + \lambda^{2}\int_{0}^{t} y(\tau) d\tau, \quad t>0,
 \end{eqnarray}
thus for $t=1$ we get the integral boundary condition.

Finally, from (\ref{eq17}) we have

 \begin{eqnarray}\label{eq18}
  y^{(3)}(t)= {\lambda\over \sqrt{\pi}} t^{-3/2} + \lambda^{2} y(t) , \quad t>0,
 \end{eqnarray}

 so the singular ordinary differential equation (\ref{edo}) holds, thus the  proof of the theorem
 is complet.
\end{proof}

\begin{theorem}\label{th3}
 The solution of the Volterra integral equation {\rm(\ref{Vol})} is given by the following expression
  \begin{eqnarray}\label{eq19}
  y(t) = I(t) - \sqrt{{2\over \pi}} J(t),  \quad t>0,
   \end{eqnarray}
with
 \begin{eqnarray}\label{eq20}
  I(t)= \sum_{n=0}^{+\infty} { (\lambda^{2/3} t)^{3n}\over (3n)!}
 \end{eqnarray}

  \begin{eqnarray}\label{eq21}
 J(t)= \sum_{n=0}^{+\infty} {(2 \lambda^{2/3} t)^{{3(2n+1)} \over 2}\over (3(2n+1))!!}
  \end{eqnarray}
  are series with infinite radii of convergence and we use the definition
  $$  (2n+1)!! = (2n+1)(2n-1)(2n-3)\cdots  5\cdot 3\cdot 1.$$
for compactness expression.
\end{theorem}
\begin{proof}
 By using the Adomian method \cite{Adm, waz1} we propose, for the solution of the Volterra integral equation
 (\ref{Vol}), the following serie of expansion functions given by
 $$  y(t) = \sum_{n=0}^{+\infty} y_{n}(t)$$
 and we obtain the following recurrence expansions :
 $$ y_{0}(t) = 1, \quad y_{n}(t) =  -{2\lambda\over \sqrt{\pi}}\int_{0}^{t} y_{n-1}(\tau) \sqrt{t-\tau} \, d\tau,  \quad \forall n\geq 1.$$
Then, following \cite{BoTA2016a} we obtain (\ref{eq19}) where $I(t)$ and $J(t)$ are given by (\ref{eq20}) and   (\ref{eq21})
 respectively, and the result holds.

 \bigskip

 The solution of the Volterra integral equation (\ref{Vol}) is the key in order to obtain the solution
 of the following nonclassical heat conduction problem given by

 \begin{eqnarray}
  u_{t}(x, t) - u_{xx}(x , t) &=& -\lambda \int_{0}^{t} u_{x}(0, \tau) d\tau, \quad x>0, \quad t>0,\label{ed}
  \\
  u(0, t) &=& 0,  \quad t>0,\label{bc}
  \\
  u(x , 0) &=& h_{0} >0, \quad x >0\label{ic},
 \end{eqnarray}
with a parameter $\lambda \in \br$.
 Then the solution of the problem above is given by
 \begin{eqnarray}\label{solu}
   u(x , t) = h_{0} erf\left( {x\over 2\sqrt{t}}\right)  - \lambda\int_{0}^{t} erf\left({x\over 2\sqrt{t-\tau}}\right) U(\tau) d\tau
  \end{eqnarray}
 where $U(t)$ is given by
 \begin{eqnarray}\label{U1}
 U(t)= {h_{0}\over \sqrt{\pi}} \int_{0}^{t} {g(\tau)\over \sqrt{t-\tau}} d\tau
  \end{eqnarray}
 and $g$ is the solution of the Volterra integral equation (\ref{Vol}). Moreover, the heat flux on $x=0$ is given by
 $$
 u_{x}(0, t)= U'(t)= {h_{0}\over \sqrt{\pi t}}-  h_{0} \lambda \int_{0}^{t} g(\tau) d\tau, \quad t>0.$$

 For the complet proof see \cite{BoTA2016a}.
\end{proof}

 \section{Dependence of the solution with respect to $\lambda$}

From now on, we will consider that the solution to the singular  ordinary differential equation of third order
with an integral boundary condition (\ref{edo}) or equivalently the  solution of the
Volterra integral equation (\ref{Vol}) depends also on the parameter $\lambda \in \br$.

We consider that $t\mapsto g_{\lambda}(t)$ be the solution of the Volterra integral equation (\ref{Vol}) for the
parameter $\lambda$.  For $\varepsilon \in (0 , 1)$ be a fixed real number and  $T>0$, let consider  the
parameter $\lambda$ such that
\begin{eqnarray}\label{lambda}
 |\lambda|\leq \lambda_{\varepsilon , T} = {3\sqrt{\pi}\over 4} {\varepsilon\over T^{3/2}},
\end{eqnarray}
and we define the norm
$$  \|g\|_{T}= \max_{0\leq t\leq T} |g(t)|.$$
Therefore, we obtain the following dependence   results.

 \begin{theorem}\label{th3.1}
  We have the boundedness
  \begin{eqnarray}\label{B}
   \|g_{\lambda}\|_{T}\leq {1\over 1-\varepsilon}, \quad \forall \lambda :  \quad |\lambda|\leq \lambda_{\varepsilon , T}.
 \end{eqnarray}
 Moreover the application $\lambda \mapsto g_{\lambda}(t)$   defined from $[-\lambda_{\varepsilon , T} ,  \lambda_{\varepsilon , T}]$,
 to ${\cal C}([0 , T))$ is Lipschitzian.
 \end{theorem}
\begin{proof}
 From the Volterra integral equation (\ref{Vol}) we obtain
  \begin{eqnarray*}\label{BB}
   |g_{\lambda}(t)|&\leq&  1+ {2|\lambda|\over \sqrt{\pi}} \|g_{\lambda}\|_{t}\int_{0}^{t}\sqrt{t-\tau} \, d\tau
   \nonumber\\
   &\leq&  1+ {4\over 3\sqrt{\pi}} \lambda_{\varepsilon , T} T^{3/2} \|g_{\lambda}\|_{T}
  \end{eqnarray*}
and by using (\ref{lambda}) follows (\ref{B}).
Moreover, consider $g_{i}(t)$ the solution of the Volterra integral equation (\ref{Vol}) for $\lambda_{i}$ (i= 1, 2)), such that
$$| \lambda_{i}|\leq  \lambda_{\varepsilon , T}.$$
 Then, we have
  \begin{eqnarray*}
   |g_{2}(t)- g_{1}(t)| &\leq & {2\over \sqrt{\pi}}|\lambda_{2}-\lambda_{1}| \|g_{1}\|_{t}\int_{0}^{t} \sqrt{t-\tau} \, d\tau
     +{2|\lambda_{2}|\over \sqrt{\pi}}   \|g_{1}- g_{2}\|_{t}\int_{0}^{t} \sqrt{t-\tau} \, d\tau
     \nonumber\\
    &\leq& {4 T^{3/2}\over 3\sqrt{\pi}} \left[  |\lambda_{2}-\lambda_{1}| \|g_{1}\|_{T}+ |\lambda_{2}| \|g_{2}- g_{1}\|_{T}\right].
  \end{eqnarray*}
Therefore, we get
 \begin{eqnarray}\label{eq3.3}
  \|g_{2}- g_{1}\|_{T} \leq {4\over 3\sqrt{\pi}}{T^{3/2}\over (1-\varepsilon)^{2}}|\lambda_{2}-\lambda_{1}|
 \end{eqnarray}
  thus the result holds.
\end{proof}

Now, we obtain the dependence of the solution to the nonclassical heat conduction problem (\ref{ed})-(\ref{ic})
with respect to the parameter $\lambda$.  We consider that $U_{\lambda}$ and $u_{\lambda}$ are  given respectively by

  \begin{eqnarray}\label{U34}
 U_{\lambda}(t)= {h_{0}\over \sqrt{\pi}}\int_{0}^{t} {g_{\lambda}(\tau)\over \sqrt{t-\tau}} d\tau
  \end{eqnarray}

 and

 \begin{eqnarray}\label{u2}
   u_{\lambda}(x, t)=  h_{0} erf\left({x\over 2\sqrt{t}} \right)-
 \lambda \int_{0}^{t} erf\left({x\over 2\sqrt{t-\tau}} \right)  U_{\lambda}(\tau) d\tau.
 \end{eqnarray}

  Then, we obtain the following results:

   \begin{theorem}\label{th3.2}
    We have the boundedness
    \begin{eqnarray}\label{3.4}
     \|U_{\lambda}\|_{T} \leq {2h_{0}\over \sqrt{\pi}}{T^{1/2}\over 1-\varepsilon},
     \quad \forall \lambda : \quad |\lambda| \leq \lambda_{\varepsilon, T}.
    \end{eqnarray}
Moreover, the application $\lambda \mapsto U_{\lambda}(t)$, from $[-\lambda_{\varepsilon, T}, \lambda_{\varepsilon, T} ]$ to
${\cal C}([0 , T])$ is Lipschitzian.
We have also the following  boundedness
 \begin{eqnarray}\label{3.6}
  \|u_{\lambda}\|_{[0 , +\infty[\times[0, T]}\leq h_{0} \left(1+  {3 \varepsilon\over 2( 1-\varepsilon)}\right)
  \quad \forall \lambda : \quad |\lambda| \leq \lambda_{\varepsilon, T},
 \end{eqnarray}
the estimates
 \begin{eqnarray}\label{3.7}
\|u_{\lambda}- u_{0}\|_{[0, +\infty[\times[0, T]}\leq   {2h_{0}\over \sqrt{\pi}}{T^{3/2}\over 1-\varepsilon}|\lambda|,
\quad \forall \lambda : \quad |\lambda| \leq \lambda_{\varepsilon, T},
\end{eqnarray}
and  that the application $\lambda \mapsto u_{\lambda}(x, t)$, from $[-\lambda_{\varepsilon, T}, \lambda_{\varepsilon, T}  ]$ to
${\cal C}([0 , +\infty[ \times  [0 , T])$ is Lipschitzian.
   \end{theorem}
\begin{proof}
 From (\ref{U1}) we have
 $$ |U_{\lambda}(t) \leq {h_{0}\over \sqrt{\pi}}\|g_{\lambda}\|_{t}\int_{0}^{t} {d\tau\over \sqrt{t-\tau}}
 \leq  {2 h_{0}\over \sqrt{\pi}}{T^{1/2}\over 1-\varepsilon} $$
 thus (\ref{3.4}) holds.
 Consider now $U_{i}(t)$ given by (\ref{U34}), for $\lambda_{i}$  ($i=1, 2$) satisfying $|\lambda_{i}|\leq \lambda_{\varepsilon, T}.$
 We have
 \begin{eqnarray}
  |U_{2}(t) -U_{1}(t)| \leq {2h_{0} T^{1/2}\over \sqrt{\pi}} \|g_{2}- g_{1}\|_{T}
  \leq {8h_{0} T^{2}\over 3\pi (1-\varepsilon)^{2}} |\lambda_{2}-\lambda_{1}|
 \end{eqnarray}
thus the application $\lambda \mapsto U_{\lambda}$  is Lipschitzian.

From (\ref{u2}) we have
\begin{eqnarray*}
 |u_{\lambda}(x, t)|\leq h_{0} + t |\lambda| \|U_{\lambda}\|_{t}
 \leq h_{0} \left(1+  {3 \varepsilon\over 2(1-\varepsilon)}\right),
 \quad \forall x\in[0 , +\infty[, \quad \forall t\in [0 , T],
\end{eqnarray*}
thus (\ref{3.6}) holds.

From (\ref{u2}) also, we have
\begin{eqnarray*}
 |u_{\lambda}(x , t) - u_{0}(x , t)| \leq t |\lambda| \|U_{\lambda}\|_{t}
 \leq {2h_{0}\over \sqrt{\pi}}{T^{3/2}\over 1-\varepsilon} |\lambda|
\end{eqnarray*}

thus (\ref{3.7}) holds.

Consider now $u_{i}(x , t)$ given by (\ref{u2}) for $\lambda_{i}$
($i=1, 2$) satisfying $|\lambda_{i}|\leq   \lambda_{\varepsilon ,
T}$.  Then, we have
\begin{eqnarray*}
 |u_{2}(x , t) - u_{1}(x , t)| &\leq&  t|\lambda_{2}-\lambda_{1}| \|U_{1}\|_{t}  +  t|\lambda_{2}|\|U_{2}-U_{1}\|_{t}
 \nonumber\\
&\leq & T |\lambda_{2}-\lambda_{1}|  \|U_{1}\|_{T} + T|\lambda_{2}|\|U_{2}-U_{1}\|_{T}
 \nonumber\\
&\leq &  {2 h_{0}  T^{1/2}\over \sqrt{\pi} (1-\varepsilon)} (T+ {\varepsilon\pi\over 1-\varepsilon}) |\lambda_{2}-\lambda_{1}|
\quad \forall x\in [0, +\infty[, \quad \forall t\in [0 ,T],
\end{eqnarray*}

thus

$$
\|u_{2}-u_{1}\|_{[0 , +\infty[\times [0 , T]}
\leq {2 h_{0}  T^{1/2}\over \sqrt{\pi} (1-\varepsilon)}\left({\varepsilon\pi \over 1-\varepsilon} +T\right)
    |\lambda_{2}-\lambda_{1}|
$$
and the result holds.
\end{proof}

\bigskip
{\bf Conclusion} We have obtained the equivalence between a family of singular  ordinary differential equations
 of third order
with an integral boundary condition (\ref{edo}) and the Volterra integral equation (\ref{Vol}) with a parameter $\lambda\in \br$. We have also given the explicit solution of these equations and then some nonclassical heat conduction problems can be solved  explicitely, for any  real parameter $\lambda$.
Finally, we have established  the dependence of the family of singular  differential equations  of third order with respect to the parameter $\lambda$.

\bigskip

\noindent{\bf Acknowledgements:}
 This paper was partially sponsored by the Institut Camille Jordan St-Etienne University for first author,
 and the projects PIP $\#$ 0275 from CONICET-Austral (Rosario, Argentina) and Grant AFOSR-SOARD FA 9550-14-1-0122
 for the second author.


\begin{thebibliography}{00}

\bibitem{Adm}{\sc G. Adomian},
{ Solving frontier Problems of Physics decomposition method},
Springer (1994).

\bibitem{AlNtMeAhAl2015}
{\sc  H.H. Alsulami, S. K. Ntouyas, S. A. Al-Mezel, B. Ahmad, A. Alsaedi},
{\it A study of third-order single-valued and multi-valued problems with integral boundary conditions},
Boundary Value Problems, 2015  No. 25 (2015), 1-30.

\bibitem{BeBoBo2008}
{\sc Z. Benbouziane, A. Boucherif, S.M. Bouguima},
{\it Existence result for impulsive third order periodic boundary value problems},
Appl. Math. Comput., 206 (2008), 728-737.

\bibitem{BeTaVi2000} {\sc L.R.Berrone, D.A.Tarzia, L.T.Villa},
{\it Asymptotic behavior of a non-classical heat conduction problem for
a semi-infinite material},
 Math. Meth. Appl. Sci., 23 (2000), 1161-1177.

\bibitem{Bo2009}
{\sc  A. Boucherif},
{\it Second-order boundary value problems with integral boundary conditions},
Nonlinear Anal.,  70 (2009), 364-371.

\bibitem{BoBoMaBe2009}
{\sc  A. Boucherif, S.M. Bouguima, N. Al-Malki, Z. Benbouziane},
{\it  Third order differential equations with integral boundary conditions},
Nonlinear Anal.,  71 (2009),  e1736-e1743.

\bibitem{BoBoBeMa2014}
{\sc A. Boucherif, S.M. Bouguima, Z. Benbouziane, N. Al-Malki},
{\it Third order problems with nonlocal conditions of integral type},
Boundary Value Problems,  2014 No. 137 (2014), 1-10.

\bibitem{BoTA2016b}
{\sc M.Boukrouche, D.A. Tarzia},
{\it  Global solution to a non-classical
        heat problem in the semi-space $\br^{+}\times\br^{n-1}$},
Quart. Appl. Math., 72 (2014), 347-361.

\bibitem{BoTA2016a}{\sc M.Boukrouche, D.A. Tarzia},
{\it A nonclassical heat conduction problem with non local source},
Boundary Value Problems, (2017), In Press.

\bibitem{CaLi1990} {\sc J.R. Cannon, Y. Lin},
{\it A Galerkin procedure for diffusion equations with boundary integral conditions},
Int. J. Engng. Sci., Vol. 28 No. 7 (1990), 579-587.

\bibitem{CeTaVi2015}{\sc A.N. Ceretani, D.A. Tarzia, and L.T. Villa},
{\it Explicit solutions for a non-classical heat conduction problem for a semi-infinite strip
with a non-uniform heat source}, Boundary Value Problems, 2015 No. 156 (2015), 1-26.


\bibitem{Ch2015} {\sc M. Cheng},
{\it Nagumo theorems of third-order singular nonlinear boundary value problems},
Boundary Value Problems, 2015 No. 135 (2015), 1-11.

\bibitem{DaHu2007}{\sc D.Q. Dai, Y.Huang},
{\it Remarks on a semilinear heat equation with integral boundary conditions},
Nonlinear Anal., 67 (2007), 468-475.

\bibitem{De2005} {\sc M. Dehghan},
{\it Efficient techniques for the second-order parabolic equation subject to nonlocal specifications},
Appl. Numer. Math., 52 (2005), 39-62.

\bibitem{De2007} {\sc M. Dehghan},
{\it The one-dimensional heat  equation subject to   a   boundary integral  specifications},
Chaos, Solitons Fractals,  32 (2007), 661-675.


\bibitem{Du2007} {\sc Z. Du},
{\it Singularly perturbed third-order boundary value problem for nonlinear systems},
Appl. Math. Comput., 189 (2007), 869-877.

\bibitem{Du2011} {\sc Z. Du},
{\it Existence and unqueness results for third-order  nonlinear differential systems},
Appl. Math. Comput., 218 (2011), 2981-2987.

\bibitem{DuGeLi2004}{\sc Z. Du, W. Ge, X. Lin},
{\it Existence of solutions for a class of third-order nonlinear boundary value problems}
J. Math. Anal. Appl., 294 (2004), 104-112.

\bibitem{DuGeZh2005} {\sc Z. Du, W. Ge, M. Zhou},
{\it Singular perturbations for third-order nonlinear multi-point boundary value problem},
J. Diff. Eq., 218 (2005), 69-90.

\bibitem{GrKo2009}{\sc  J. R. Graef, L. Kong},
{\it  Positive solutions for third order semipositone boundary  value problems},
Appl. Math. Letters, 22 (2009), 1154-1160.

\bibitem{GuLiLi2012}{\sc Y. Guo, Y. Liu, Y. Liang},
{\it Positive solutions for the third-order boundary value problems with the second derivatives},
Boundary Value Problems,  2012 No. 34 (2012), 1-9.

\bibitem{HaThLeIv2014} {\sc D. N. Hao, P.X. Thanh, D. Lesnic,   M. Ivanchov},
 {\it Determination of a source in the heat equation from integral observations},
 J. Comput. Appl. Math., 264 (2014), 82-98.

\bibitem{HeLu2016}{\sc H. Henderson, R. Luca},
{\it Positive solutions for a system of semipositone coupled fractional boundary value problems},
 Boundary Value Problems, 2016 No. 61 (2016), 1-23.

\bibitem{JiLi2007}
  {\sc W. Jiang, F. Li},
  {\it Several existence theorem of monotone positive solutions for third-order  multi-point boundary value problems},
  Boundary Value Problems, 2007, Art. ID 17951, 1-9.

\bibitem{JiGuYa2015} {\sc Y. Ji, Y.Guo, Y. Yao},
 {\it Positive solutions for higher order differential equations with integral boundary conditions},
 Boundary Value Problems, 2015 No. 214 (2015),  1-11.

\bibitem{Ka2013} {\sc F. Kanca},
{\it The inverse problem of the heat equation  with periodic boundary and integral overdetermination conditions},
J. Inequalities Appl., 2013 No. 108 (2013), 1-9.

\bibitem{Lin2013} {\sc X.Lin},
{\it Singular perturbations of third-order nonlinear differential equations
with full nonlinear boundary conditions},
Appl. Math. Comput., 224 (2013), 88-95.

\bibitem{LiDoLi2008}{\sc X. Lin, Z. Du, W. Liu},
{\it Uniqueness and existence results for a third-order nonlinear multi-point boundary value problem},
Appl. Math. Comput., 205 (2008), 187-196.

\bibitem{LiUmAnKo2007} {\sc Z. Liu, J.S. Ume, D.R. Anderson, S.M. Kang},
{\it Twin monotone positive solutions to a singular nonlinear third-order differential equation},
J. Math. Anal. Appl., 334 (2007), 299-313.

\bibitem{Li1999} {\sc Y. Liu},
{\it Numerical solution of the heat equation with nonlocal boundary conditions},
J. Comput. Appl. Math., 110 (1999), 115-127.

\bibitem{LiLILi2014} {\sc B. Liu, J. LI, L. Liu},
{\it Nontrivial solutions for a boundary value problem with integral boundary conditions},
Boundary Value Problems,  2014 No. 15 (2014), 1-8.

\bibitem{LiLiWu2015}{\sc H. Li, L. Liu,  Y. Wu},
{\it Positive solutions for singular nonlinear fractional differential equation with integral boundary conditions},
 Boundary Value Problems, 2015 No. 232 (2015), 1-15.

\bibitem{LvFe2016}{\sc G. Luand, M. Feng},
{\it Positive Green's function and triple positive solutions of a second-order impulsive differential equation
with integral boundary conditions and delayed argument},
Boundary Value Problems,  2016 No. 88 (2016), 1-17.

\bibitem{Ma2009} {\sc J. Martin-Vaquero},
{\it Two-level forth-order explicit schemes for diffusion equations Subject to boundary integral specifications},
Chaos Solitons Fractals, 42 (2009), 2364-2372.

\bibitem{MaVi2009A} {\sc J.  Martin-Vaquero, J.Vigo-Aguiar},
{\it A note on efficient techniques for the second-order parabolic equation subject to non-local conditions},
Appl. Numer. Math., 59 (2009), 1258-1264.

\bibitem{MaVi2009B} {\sc J.  Martin-Vaquero, J.Vigo-Aguiar},
{\it On the numerical soliton of the heat conduction equations subject to nonlocal conditions},
Appl. Numer. Math., 59 (2009), 2507-2514.

\bibitem{MaWa2012} {\sc J. Martin-Vaquero, B.A. Wade},
{\it On efficient numerical methods for an initial-boundary value problem with nonlocal boundary conditions},
Appl. Math. Modelling,  36 (2012), 3411-3418.

\bibitem{MeBo2007} {\sc N. Merazga, A. Bouziani},
{\it On a time-discretization method  for a semilinear heat equation with purely integral conditions in a
nonclassical function space}, Nonlinear Anal., 66 (2007), 604-623.

\bibitem{PaXiCa2015}{\sc H. Pang, W. Xie, L. Cao},
{\it   Successive iteration  and positive solutions for a third-order boundary value
problem involving integral conditions},
Boundary Value Problems,  2015 No. 139 (2015), 1-10.


\bibitem{SaTaVi2011}{\sc N.N. Salva, D.A. Tarzia, and L.T. Villa},
{\it  An initial-boundary value problem for the one-dimensional
non-classical heat equation in a slab},
 Boundary Value Problems, 2011 No. 4 (2011), 1-17.


\bibitem{SuXi2016}{\sc M. Sun, Y. Xing},
{\it Existence results for a king of forth-order  impulsive integral boundary value problems},
Boundary Value Problems, 2016 No. 81 (2016), 1-15.

\bibitem{SuLi2010}{\sc  J. P. Sun, H.B. Li},
{\it monotone positive solution of nonlinear third-order BVP with integral boundary conditions},
Boundary Value Problems,  2010 (2010), Art. ID 874959, 1-12.


\bibitem{TaVi1998} {\sc D.A.Tarzia, L.T.Villa},
{\it Some nonlinear heat conduction problems for a semi-infinite
 strip with a non-uniform heat source},
 Rev. Uni\'on Mat. Argentina, 41 (1998), 99-114.


 \bibitem{ToDi2016} {\sc Z. Tong, W. Ding},
 {\it Existence of symmetric solutions for a class of BVP with integral boundary conditions},
 Boundary Value Problems, 2016 No. 84 (2016),  1-11.


\bibitem{Vi1986} {\sc L.T. Villa},
{\it  Problemas de control para una ecuaci\'on unidimensional no homog\'enea del
  calor}, Rev. Uni\'on Mat. Argentina, 32 (1986), 163-169.


\bibitem{WaGu2016}{\sc W. Wang, X. Guo},
{\it Eigenvalue problem for fractional differential equations with nonlinear integral and disturbance parameter
in boundary conditions},
Boundary Value Problems,  2016 No. 42 (2016),  1-23.

\bibitem{WaGe2007}{\sc Y. Wang, W. Ge},
{\it  Existence of solutions for a third order differential equation with integral boundary conditions},
Computers Math. Appl.,   53 (2007), 144-154.

\bibitem{waz1} {\sc A.M. Wazwaz},
    Linear and nonlinear  integral equations.
    Methods and applications, Springer Heidelberg (2011).

\bibitem{Zh2014} {\sc H.E. Zhang},
{\it   Muliple positive solutions of nonlinear BVPs for differential systems involving integral conditions},
Boundary Value Problems, 2014 No. 61 (2014), 1-13.

\bibitem{Zh2011} {\sc P. Zhang},
{\it Iterative solutions of singular boundary value problems of third-order differential equation},
Boundary Value Problems, 2011 (2011), Art. ID 483057, 1-10.

\bibitem{ZhGe2009-1} {\sc X. Zhang, W. Ge},
{\it Positive solutions for a class of boundary-value problems with integral boundary conditions},
Computers Math. Appl., 58 (2009), 203-215.

\bibitem{ZhFe2015}{\sc  X. Zhang, M. Feng},
{\it Positive solutions for a second-order differential equation with integral  boundary  conditions and deviating arguments},
Boundary Value Problems,  2015 No. 222 (2015), 1-21.

\bibitem{ZhFeGe2009}{\sc X. Zhang,  M. Feng, W. Ge},
{\it Existence result of second-order  differential equations with integral boundary conditions and resonance},
 J. Math. Anal. Appl., 352 (2009), 311-319.

\bibitem{ZhXu2016}{\sc L. Zhang, Z. Xuan},
{\it Muliple positive solutions for a second-order boundary value problem with integral boundary conditions},
 Boundary Value Problems,  2016 No.60 (2016), 1-8.
 \end{thebibliography}
\end{document}